\documentclass{article}
\usepackage[margin=3.5cm]{geometry}
\usepackage[utf8]{inputenc}

\usepackage{amsmath,amssymb,amsfonts,amsthm}
\usepackage{multirow}

\usepackage{hyperref}
\usepackage{tikz-cd}

\usepackage{import}
\graphicspath{{Figures/}}
\usepackage[margin=20pt]{caption} 
\usepackage{graphicx}

\usepackage[style=alphabetic,sorting=nyt,maxbibnames=7,maxcitenames=5,block=space,backend=biber]{biblatex}
\newtheorem{thm}{Theorem}[section]
\newtheorem{prop}[thm]{Proposition}
\newtheorem{cor}[thm]{Corollary}
\newtheorem{lem}[thm]{Lemma}
\newtheorem{thm*}{Theorem}

\theoremstyle{definition}
\newtheorem{dfn}[thm]{Definition}
\newtheorem{ex}[thm]{Example}
\AtEndEnvironment{ex}{\null\hfill\ensuremath{\triangle}}

\newcommand{\into}{\hookrightarrow}
\newcommand{\onto}{\twoheadrightarrow}
\renewcommand{\SS}{\mathbb{S}}
\newcommand{\NN}{\mathbb N}
\newcommand{\ZZ}{\mathbb Z}
\newcommand{\RR}{\mathbb R}

\newcommand{\V}{\operatorname{V}}
\newcommand{\E}{\operatorname{E}}

\newcommand{\K}{\mathrm K_{3,3}}
\newcommand{\M}{\mathrm M}
\renewcommand{\L}{\mathcal L}
\newcommand{\C}{\mathcal C}

\newlength{\picsize}
\title{Writhe invariants \\of 3-regular spatial graphs}
\author{Stefan Friedl$^1$ \and Tejas Kalelkar$^2$ \and José Pedro Quintanilha$^3$}

\date{$^1$Universität Regensburg\\
	$^2$Indian Institute of Science Education and Research Pune
	$^3$Ruprecht-Karls-Universität Heidelberg\\\bigskip
	\today}

\begin{document}
\maketitle

\begin{abstract}
	We give a necessary condition for two diagrams of 3-regular spatial graphs with the same underlying abstract graph~$G$ to represent isotopic spatial graphs. The test works by reading off the writhes of the knot diagrams coming from a collection of cycles in~$G$ in each diagram, and checking whether the writhe tuples differ by an element in the image of a certain map of $\ZZ$-modules determined by~$G$. We exemplify by using our result to distinguish, for each $n\ge 3$, all elements in a certain infinite family of embeddings of the Möbius ladder~$\M_n$ into~$\RR^3$. We also connect these writhe tuples to a classical invariant of spatial graphs due to Wu and Taniyama.
\end{abstract}

\section{Introduction}

The classical question in knot theory of determining whether two knots or links are isotopic naturally leads to the generalization where we consider instead embeddings of (finite) graphs into~$\RR^3$, that is, spatial graphs. An algorithm for determining whether two spatial graphs are isotopic has recently been published \cite{FMQS24}, but its computational complexity sets it beyond the realm of practical applicability.

One can sometimes obstruct the existence of an isotopy between two given spatial graphs by distinguishing the multisets of links contained in them as unions of cycles; in other words, reducing the problem to one of classical knot theory. However, many examples resist this technique -- consider, for instance, the embeddings of the complete bipartite graph on~$3+3$ vertices~$\K$ depicted in Figure~\ref{fig:mainexample_simplified}.
\begin{figure}[]
	\centering
	\def \svgwidth{0.55\linewidth}
	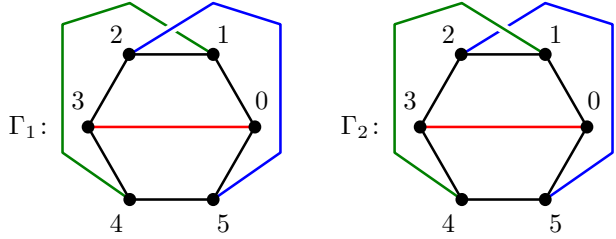
	\caption{Two spatial graphs of type~$\K$.}\label{fig:mainexample_simplified}
\end{figure}
Here, all links that can be found as subgraphs of each embedding are unknots. Moreover, these spatial graphs are mirror images of one another, which typically makes the task of distinguishing them more difficult.

In this note, we give a rather elementary test for comparing $3$-regular spatial graphs (meaning, where all vertices have degree~$3$), which works in particular for distinguishing the pair in Figure~\ref{fig:mainexample_simplified}. The starting point is the notion of the writhe~$w(D)$ of a diagram~$D$ for a knot~$K$, that is, the signed count of crossings in~$D$. It is well-known that~$w(D)$ carries no information about~$K$ -- one can easily produce a diagram of~$K$ with arbitrary writhe. However, given a spatial graph~$\Gamma$, we can fix a diagram~$D$ for~$\Gamma$ and simultaneously consider the writhes of all knot diagrams corresponding to cycles in the abstract underlying graph~$G$ of~$\Gamma$. Our main result constrains the possible values of such a ``writhe tuple''. More precisely, denoting by~$\E(G)$ the edge set of~$G$, we associate to each set~$C$ of cycles in~$G$ a $\ZZ$-linear map 
\[C^* \colon \ZZ^{\E(G)} \to  \ZZ^C,\]
and relate the possible writhe-tuples~$w_C(D)$ of diagrams~$D$ of~$\Gamma$ as follows:

\begin{thm*}[Invariance of ${[w_C(D)]}$]\label{thm:main_intro}
	 Let $\Gamma_1, \Gamma_2$ be $3$-regular spatial graphs with the same underlying abstract graph~$G$, and let $C$~be a set of cycles of~$G$. Let $D_1, D_2$ be diagrams of~$\Gamma_1, \Gamma_2$, respectively, with matching cyclic orders. If $\Gamma_1$ is isotopic to~$\Gamma_2$, then 
\[w_C(D_2) - w_C(D_1) \in  C^*(\ZZ^{\E(G)}).\]
\end{thm*}

This theorem is re-stated and proved in the text as Theorem~\ref{thm:main}. The ``matching cyclic orders'' assumption is a minor condition, easily circumvented when one wishes to apply the theorem for obstructing an isotopy between $\Gamma_1$~and~$\Gamma_2$.

We also show that for each $n\ge 3$, Theorem~\ref{thm:main_intro} can be applied to prove that the elements in a particular infinite family of embeddings of the Möbius ladder~$\M_n$ are pairwise non-isotopic (Example~\ref{ex:moebius}), vastly generalizing the example in Figure~\ref{fig:mainexample_simplified}. At the end, we explain how a slightly coarser equivalence class of writhe tuples than the one in Theorem~\ref{thm:main_intro} can be recovered from an invariant due to Wu and Tanayama \cite{Wu59, Tan95} that classifies spatial graphs ``up to homology''.

\subsection*{Structure of this article}

Section~\ref{sec:preliminaries} lays out the terminology necessary for stating and proving our main result.
 We group these preliminaries into three subsections, pertaining to knots, graphs, and spatial graphs. Most of the material is standard and can be skipped and consulted  occasionally for clarifying our conventions, with two exceptions: Definitions \ref{dfn:cyclemap}~and~\ref{dfn:writhetuple} are not standard, and play a key role in this article.
 
In Section~\ref{sec:mainthm} we prove Theorem~\ref{thm:main_intro} and use it to show that the spatial graphs in Figure~\ref{fig:mainexample_simplified} are not isotopic. This example is then generalized to certain families of embeddings of the Möbius ladder~$\M_n$, for each $n\ge 3$.

Section~\ref{sec:extras} collects a few additional observations, namely on limitations of Theorem~\ref{thm:main_intro}, and in Section~\ref{sec:wu} we briefly explain the Wu invariant and its connection to Theorem~\ref{thm:main_intro}.

\subsection*{Acknowledgements}

We are thankful to Lars Munser and Yuri Santos Rego for several discussions. The first author was supported by the CRC 1085 ``Higher Invariants'' (Universität Regensburg, funded by the DFG).

\section{Preliminaries}\label{sec:preliminaries}

We shall work in the piecewise-linear (``PL'') category, following the conventions in Rourke-Sanderson's textbook \cite{RS72}.

\subsection{Knots}\label{sec:knots}
In this article, we will define a \textbf{knot} to be a PL subspace $K\subset \RR^3$ that is PL-homeomorphic to~$\SS^1$. We do not take the homeomorphism as part of the structure; in particular, knots will not carry orientations. We say $K$~\textbf{has a diagram} if its image under the standard projection $\pi \colon \RR^3 \to \RR^2$ has only double-point self-intersections, and they are all transverse (``\textbf{crossings}''). Then, the \textbf{diagram} of~$K$ is $\pi(K)$ together with the data, at each crossing, of which strand crosses over which.
After choosing an orientation for~$K$, one associates to each crossing in a diagram~$D$ a sign and recording whether the over-strand moves right-to-left or left-to-right of the under-strand; see Figure~\ref{fig:crossings}. Although one reads off crossing signs by choosing an orientation of~$K$, reversing this orientation does not change the sign, so the sign of a crossing is well-defined for diagrams of unoriented knots.
If $n_+, n_-$ are, respectively, the number of positive and negative crossings, then the \textbf{writhe} of~$D$ is the integer
\[w(D) := n_+ - n_-.\]

\begin{figure}[h]
	\centering
	\def \svgwidth{0.2\linewidth}
	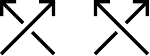
	\caption{A positive crossing (left) and a negative crossing (right).}
	\label{fig:crossings}
\end{figure}

Two knots $K_1, K_2$ are \textbf{isotopic} if there is a PL isotopy from $\mathrm{id}_{\RR^3}$ to a PL self-homeo\-mor\-phism~$\Phi$ of~$\RR^3$ with $\Phi(K_1) = K_2$.
It is well-known that the diagram of a knot~$K$ determines~$K$ up to isotopy, and every knot is isotopic to one that has a diagram. As we are interested in knots only up to isotopy, we loosely speak of ``a diagram of~$K$'', even when $K$~does not have a diagram, to mean ``the diagram of a knot isotopic to~$K$''.

As Figure \ref{fig:changewrithe} illustrates, one can easily find isotopic knots with diagrams of different writhes. In fact, this trick of ``adding a kink'' to a diagram allows one to find isotopic knots whose diagrams realize any desired integer. The writhe of a diagram is thus very far from being a knot invariant.

\begin{figure}[h]
	\centering
	\def \svgwidth{0.6\linewidth}
\begingroup%
  \makeatletter%
  \providecommand\color[2][]{%
    \errmessage{(Inkscape) Color is used for the text in Inkscape, but the package 'color.sty' is not loaded}%
    \renewcommand\color[2][]{}%
  }%
  \providecommand\transparent[1]{%
    \errmessage{(Inkscape) Transparency is used (non-zero) for the text in Inkscape, but the package 'transparent.sty' is not loaded}%
    \renewcommand\transparent[1]{}%
  }%
  \providecommand\rotatebox[2]{#2}%
  \newcommand*\fsize{\dimexpr\f@size pt\relax}%
  \newcommand*\lineheight[1]{\fontsize{\fsize}{#1\fsize}\selectfont}%
  \ifx\svgwidth\undefined%
    \setlength{\unitlength}{387.3767186bp}%
    \ifx\svgscale\undefined%
      \relax%
    \else%
      \setlength{\unitlength}{\unitlength * \real{\svgscale}}%
    \fi%
  \else%
    \setlength{\unitlength}{\svgwidth}%
  \fi%
  \global\let\svgwidth\undefined%
  \global\let\svgscale\undefined%
  \makeatother%
  \begin{picture}(1,0.29711944)%
    \lineheight{1}%
    \setlength\tabcolsep{0pt}%
    \put(0,0){\includegraphics[width=\unitlength,page=1]{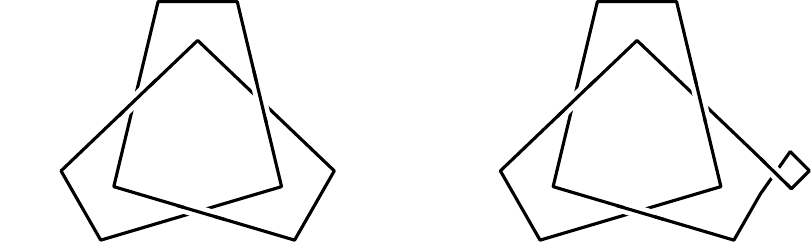}}%
    \put(-0.00335793,0.14488853){\color[rgb]{0,0,0}\makebox(0,0)[lt]{\lineheight{0}\smash{\begin{tabular}[t]{l}$D_1\colon$\end{tabular}}}}%
    \put(0.53236316,0.14488853){\color[rgb]{0,0,0}\makebox(0,0)[lt]{\lineheight{0}\smash{\begin{tabular}[t]{l}$D_2\colon$\end{tabular}}}}%
  \end{picture}%
\endgroup%

	\caption{Diagrams of isotopic knots with $w(D_1) = 3$ and $w(D_2)=4$.}
	\label{fig:changewrithe}
\end{figure}

\subsection{Graphs}\label{sec:graphs}

A \textbf{graph} will be understood as the data of a set~$\V(G)$ of \textbf{vertices}, a set~$\E(G)$ of \textbf{edges}, and a pair of functions $s_G, t_G \colon \E(G) \to \V(G)$. 
Thus we allow \textbf{loops} (that is, edges~$e$ with $s_G(e) = t_G(e)$) and multiple edges between the same two vertices, and edges that are not loops carry an orientation. 
The set of \textbf{half-edges} of~$G$ is the product $\E(G) \times\{0,1\}$. A half-edge~$(e,0)$ (respectively $(e,1)$) is \textbf{incident} to a vertex~$v$ if $s_G(e)=v$ (respectively $t_G(e)=v$).
The \textbf{degree} of~$v$ is the number of half-edges incident to~$v$; in other words:
\[\operatorname{deg}(v) := \# s_G^{-1}(v) + \# t_G^{-1}(v) \in \NN\cup\{\infty\}.\]
Given $n \in \NN$, we say $G$~is \textbf{$n$-regular} if all vertices of~$G$ have degree~$n$.

The \textbf{topological realization} of~$G$ is the space
\[ |G| := \bigg( \V(G)\sqcup  \coprod_{e\in \E(G)} [0,1]_e \bigg) / \sim, \]
where $s_G(e) \sim 0_e$ and $t_G(e)\sim 1_e$ for every~$e\in \E(G)$.
We are interested exclusively in the situation where $G$~is finite, so $|G|$~embeds as a PL subspace in~$\RR^3$. Any two PL structures on~$|G|$ are PL-homeomorphic restricting to the identity on~$\V(G)$, and respecting the orientations of all topological edges~$[0,1]_e$, so from now on we always understand~$|G|$ to carry a canonical polyhedral structure.

A \textbf{cycle} of length~$k\in\NN_{\ge1}$, or \textbf{$k$-cycle}, in~$G$ is a sub-graph comprised of $k$~distinct vertices $v_1, \ldots, v_k$ and $k$~distinct edges $e_1, \ldots, e_k$, such that for each $i\in \ZZ / k$, we have $\{s_G(e_i), t_G(e_i)\} =\{v_i, v_{i+1}\}$. 
This definition is independent of the orientation of the edges.

\begin{dfn}\label{dfn:cyclemap}
	Given a finite set~$C$ of cycles in a graph~$G$, we define the map of free $\ZZ$-modules
	\begin{align*}
		C^* \colon \ZZ^{\E(G)} &\to \ZZ^{C}\\
		e &\mapsto \sum_{c \ni e} c.
	\end{align*}
\end{dfn}

\begin{ex}\label{ex:cyclemap}
	Figure~\ref{fig:K33} depicts the complete bipartite graph $G = \K$ on $3+3$ vertices.  We refrain from specifying edge orientations, as they will not play a role in what follows. $\K$~has $9$~edges and $15$~cycles: $9$~of length~$4$, and $6$~of length~$6$. We present the matrix of~$C^*$ in the standard bases, for $C$~the set of $6$-cycles. 
\begin{figure}[h]
	\centering
	\def \svgwidth{0.2\linewidth}
\begingroup%
  \makeatletter%
  \providecommand\color[2][]{%
    \errmessage{(Inkscape) Color is used for the text in Inkscape, but the package 'color.sty' is not loaded}%
    \renewcommand\color[2][]{}%
  }%
  \providecommand\transparent[1]{%
    \errmessage{(Inkscape) Transparency is used (non-zero) for the text in Inkscape, but the package 'transparent.sty' is not loaded}%
    \renewcommand\transparent[1]{}%
  }%
  \providecommand\rotatebox[2]{#2}%
  \newcommand*\fsize{\dimexpr\f@size pt\relax}%
  \newcommand*\lineheight[1]{\fontsize{\fsize}{#1\fsize}\selectfont}%
  \ifx\svgwidth\undefined%
    \setlength{\unitlength}{141.95859534bp}%
    \ifx\svgscale\undefined%
      \relax%
    \else%
      \setlength{\unitlength}{\unitlength * \real{\svgscale}}%
    \fi%
  \else%
    \setlength{\unitlength}{\svgwidth}%
  \fi%
  \global\let\svgwidth\undefined%
  \global\let\svgscale\undefined%
  \makeatother%
  \begin{picture}(1,0.83940431)%
    \lineheight{1}%
    \setlength\tabcolsep{0pt}%
    \put(0,0){\includegraphics[width=\unitlength,page=1]{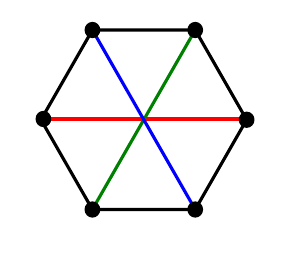}}%
    \put(0.89288958,0.41084337){\color[rgb]{0,0,0}\makebox(0,0)[lt]{\lineheight{0}\smash{\begin{tabular}[t]{l}$0$\end{tabular}}}}%
    \put(-0.00526369,0.41084337){\color[rgb]{0,0,0}\makebox(0,0)[lt]{\lineheight{0}\smash{\begin{tabular}[t]{l}$3$\end{tabular}}}}%
    \put(0.2271996,0.79123611){\color[rgb]{0,0,0}\makebox(0,0)[lt]{\lineheight{0}\smash{\begin{tabular}[t]{l}$2$\end{tabular}}}}%
    \put(0.67099292,0.79123611){\color[rgb]{0,0,0}\makebox(0,0)[lt]{\lineheight{0}\smash{\begin{tabular}[t]{l}$1$\end{tabular}}}}%
    \put(0.2271996,0.00931791){\color[rgb]{0,0,0}\makebox(0,0)[lt]{\lineheight{0}\smash{\begin{tabular}[t]{l}$4$\end{tabular}}}}%
    \put(0.67099292,0.00931791){\color[rgb]{0,0,0}\makebox(0,0)[lt]{\lineheight{0}\smash{\begin{tabular}[t]{l}$5$\end{tabular}}}}%
  \end{picture}%
\endgroup%
\medskip
	
	\begin{tabular}[c]{cccccccccc}
		&\includegraphics[width=\picsize]{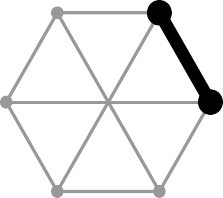}
		&\includegraphics[width=\picsize]{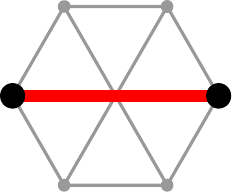}
		&\reflectbox{\includegraphics[width=\picsize, angle=180, origin=c]{e01}}
		&\includegraphics[width=\picsize]{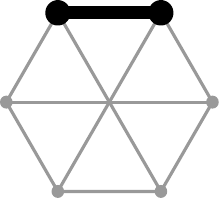}
		&\includegraphics[width=\picsize]{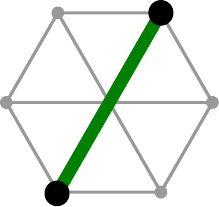}
		&\reflectbox{\includegraphics[width=\picsize]{e01}}
		&\includegraphics[width=\picsize]{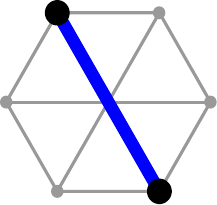}
		&\includegraphics[width=\picsize, angle=180, origin=c]{e01}
		&\includegraphics[width=\picsize, angle=180, origin=c]{e12}\\
		\includegraphics[width=\picsize]{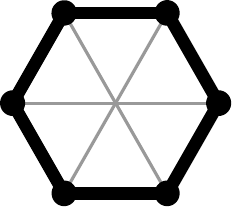}&  1&  &  1&  1&  &  1& &  1& 1 \\
		\includegraphics[width=\picsize]{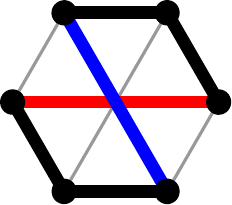}&  1& 1 &  &  1&  &  &  1&  1&  1\\
		\includegraphics[width=\picsize]{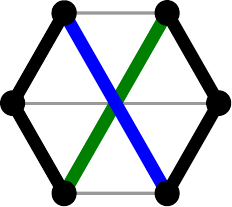}&  1&  &  1&  &  1&  1&  1&  1&  \\
		\includegraphics[width=\picsize]{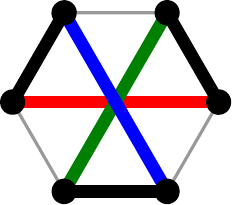}&  1&  1&  &  &  1&  1&  1&  &  1\\
		\includegraphics[width=\picsize]{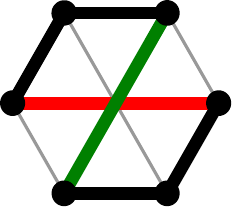}&  &  1&  1&  1&  1&  1& &  &  1\\
		\includegraphics[width=\picsize]{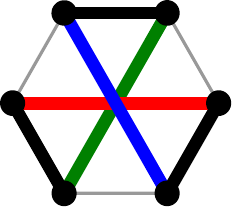}&  &  1&  1&  1&  1&  &  1& 1 &  \\
	\end{tabular}
	\caption{The graph $\K$ and the matrix of $C^*$, with $C$~the set of $6$-cycles. Some edges are colored for readability.}
	\label{fig:K33}
\end{figure}
\end{ex}

\subsection{Spatial graphs}\label{sec:spatialgraphs}

	Let a~$G$ be a finite graph. A \textbf{spatial graph} of type~$G$ is a PL embedding $\Gamma \colon |G| \into \RR^3$. 
The \textbf{support} of~$\Gamma$ is the set $|\Gamma| := \Gamma(|G|)$.
Denoting by $\pi \colon \RR^3 \onto \RR^2$ the standard projection, we say $\Gamma$~\textbf{has a diagram} if
for the standard projection $\pi\colon\RR^3 \to \RR^2$, the composition $\pi \circ \Gamma$ is injective, except for finitely-many transverse double point intersections (``\textbf{crossings}'') between interior points of edges.

Then, the \textbf{diagram} of~$\Gamma$ is the following data:
\begin{enumerate}
	\item for each $v\in \V(G)$, the point $\pi(\Gamma(v))$,
	\item for each $e\in \E(G)$, the PL-immersed curve $\pi(\Gamma([0,1]_e))$, alongside the orientation inherited from~$[0,1]_e$,
	\item at each \textbf{crossing} $p\in \RR^2$, the datum of which of the two strands meeting at~$p$ is over-crossing.
\end{enumerate}

We often use the same terminology to refer to the vertices/edges of~$G$, to the $\Gamma$-images of their topological realizations, and to their $\pi$-projections.

The spatial graph~$\Gamma$ associates to each half-edge $(e,0)$ (resp. $(e,1)$) of~$G$ the arc $\Gamma([0,\frac 12]_e)$ (resp. $\Gamma([\frac 12, 1]_e)$) in~$\RR^3$.
For each $v \in \V (G)$, the diagram~$D$ induces a cyclic ordering of these arcs, obtained by reading them off counterclockwise as their projections emanate from~$v$, which in turn cyclically orders the half-edges. Two diagrams~$D_1, D_2$ of spatial graphs of the same type~$G$ have \textbf{matching cyclic orders} if for every vertex~$v\in \V(G)$ they induce the same cyclic orders of the half-edges at~$v$.

Two spatial graphs $\Gamma_1, \Gamma_2$ of type~$G$, are \textbf{isotopic}, written $\Gamma_1 \cong\Gamma_2$, if there is a PL isotopy of~$\RR^3$ from~$\mathrm{id}_{\RR^3}$ to a map~$\Phi$ with $\Phi\circ \Gamma_1 = \Gamma_2$.
A diagram determines a spatial graph~$\Gamma$ up to isotopy, and one can always isotope a spatial graph~$\Gamma$ to one with a diagram, so we will often say ``a diagram of~$\Gamma$'' to mean ``the diagram of a spatial graph isotopic to~$\Gamma$''.

Kauffman has given the following characterization of when two diagrams represent isotopic spatial graphs:

\begin{thm}[Reidemeister moves for spatial graphs { \cite[Theorem~2.1]{Kau89}}]\label{thm:reidemeister}
	Any two diagrams representing isotopic spatial graphs differ by plane isotopies and a finite sequence of moves in five explicit types, illustrated in Figure~\ref{fig:reidemeister} for the $3$-regular case.\footnote{We include only one variant of R3, as the second can be deduced from this one and R2.}
\end{thm}

\begin{figure}[h]
	\centering
	\def \svgwidth{0.9\linewidth}
	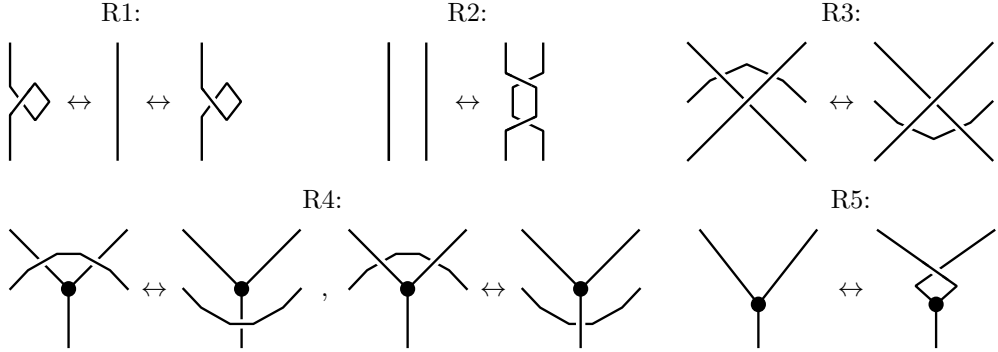
	\caption{Reidemeister moves for $3$-regular spatial graphs.}
	\label{fig:reidemeister}
\end{figure}

Evidently, the converse statement also holds: diagrams related by a sequence of these moves represent isotopic spatial graphs.

Moves R1 -- R3 are the incarnations of the classical Reidemeister moves for knots. Moves R4 and R5 involve a vertex of~$\Gamma$, which in general might have any degree. We illustrate the degree~$3$ case because our main result concerns $3$-regular spatial graphs. Note that the induced cyclic ordering on the half-edges at a vertex is only affected by moves of type R5 at that vertex.

\subsubsection*{Writhe tuples}

A spatial graph~$\Gamma$ of type~$G$ specifies, for each cycle~$c$ of~$G$, a knot~$\Gamma(|c|)$, and 
a diagram~$D$ of~$\Gamma$ yields a diagram~$D_c$ of that knot.

\begin{dfn}\label{dfn:writhetuple}
	Let $\Gamma$ be a spatial graph of type~$G$, and $D$~a diagram of~$\Gamma$. Given a set~$C$ of cycles in~$G$, the \textbf{writhe tuple} of~$D$ associated to~$C$ is
	\[w_C(D) := (w(D_c))_{c\in C} \in \ZZ^C.\]
\end{dfn}
We emphasize that although the edges of~$\Gamma$ are oriented, these orientations play no role in computing~$w_C(D)$. The signs of the crossings to be tallied into each entry~$w(D_c)$ are given by choosing an orientation for the knot~$\Gamma(|c|)$ and using it for attaching signs to the crossings in~$D_c$. Example~\ref{ex:writhetuple} shows a crossing in a spatial graph diagram~$D$ that is counted with different signs depending on which~$D_c$ is being considered.

\begin{ex}\label{ex:writhetuple}
Figure~\ref{fig:mainexample} depicts diagrams~$D_1, D_2$ of two spatial graphs~$\Gamma_1, \Gamma_2$ of type~$\K$, and their writhe tuples associated to the set~$C$ of $6$-cycles. 
\begin{figure}[h]
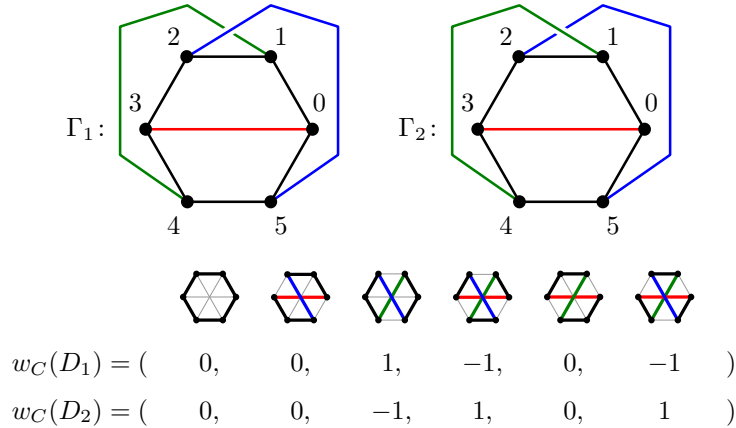

	\centering
	\def \svgwidth{0.55\linewidth}
	\input{Figures/mainexample_simplified.pdf_tex}
	\bigskip
	
	\begin{tabular}{cccccccl}
		& \includegraphics[width=\picsize]{c012345}
		& \includegraphics[width=\picsize]{c012543}
		& \includegraphics[width=\picsize]{c014325}
		& \includegraphics[width=\picsize]{c014523}
		& \includegraphics[width=\picsize]{c032145}
		& \includegraphics[width=\picsize]{c034125}& \medskip \\
		
		$w_C(D_1) = ($ & $0,$&$0,$&$1,$&$-1,$&$0,$&$-1$&$)$\medskip\\
		$w_C(D_2) = ($ & $0,$&$0,$&$-1,$&$1,$&$0,$&$1$&$)$
	\end{tabular}
	
	\caption{Diagrams for two spatial graphs of type~$\K$, and their writhe tuples associated to the set~$C$ of $6$-cycles.}
	\label{fig:mainexample}
\end{figure}
\end{ex}

Although, as we observed in Section~\ref{sec:knots}, the writhe of a knot diagram contains no information about the knot, we will show that the writhe tuple of a spatial graph diagram can contain information about the spatial graph. 

\section{Main result}\label{sec:mainthm}
In this section we shall prove Theorem~\ref{thm:main_intro}, which we re-state here for the reader's convenience:

\begin{thm}[Invariance of ${[w_C(D)]}$]\label{thm:main}
	 Let $\Gamma_1, \Gamma_2$ be $3$-regular spatial graphs of the same type~$G$, and let $C$~be a set of cycles of~$G$. Let $D_1, D_2$ be diagrams of $\Gamma_1, \Gamma_2$, respectively, with matching cyclic orders. If $\Gamma_1 \cong \Gamma_2$, then 
	 \[w_C(D_2) - w_C(D_1) \in C^*(\ZZ^{\E(G)}).\]
\end{thm}

Note that the hypothesis on the cyclic orders is no significant loss of generality. In applying Theorem~\ref{thm:main} to show that $D_1$~and~$D_2$ do not represent isotopic spatial graphs, one can first modify them with R5 moves to ensure that this condition is met.

We will prove Theorem~\ref{thm:main} by studying how the writhe tuple is affected by each Reidemeister move. The 3-regularity assumption plays a role only in describing the effect of R5-type moves.

\begin{lem}[Moves R1--R4]\label{lem:R14}
	Let $D,D'$ be diagrams of spatial graphs of type~$G$, and let $C$ be a set of cycles of~$G$.
	\begin{enumerate}
		\item If $D'$~differs from~$D$ by an $\mathrm{R1}$ move at the edge $e\in\E(G)$, then $w_C(D') = w_C(D) + \varepsilon C^*(e)$, where $\varepsilon = 1$ if the move adds a positive crossing or removes a negative crossing; otherwise $\varepsilon = -1$.
		\item If $D'$~differs from~$D$ by a move of type~$\mathrm{R2}, \mathrm{R3}$~or~$\mathrm{R4}$, then  $w_C(D') = w_C(D)$.
	\end{enumerate}
\end{lem}

Note that crossings involved in an R1 move have a well-defined sign, so the condition in item 1 makes sense independently of any choice of cycle.

\begin{proof}
	(1) Applying an R1 move at~$e$ adds the same value~$\varepsilon$ to the writhe of~$D_c$ for each~$c\in C$ containing~$e$. The writhes of diagrams associated to other cycles remain unchanged. Hence,
	\[w_C(D') -w_C(D) =  \varepsilon \sum_{c \ni e} c = \varepsilon C^*(e).\]
	
	(2) An R2 move affects only the writhes of knot diagrams~$D_c$ for cycles~$c$ containing both edges involved. In that case, the added crossings come with opposite signs, so either way $w(D_c)$~remains unaltered. Similarly, R3~merely displaces crossings between the same edges, so it has no effect on~$w(D_c)$ for any cycle~$c$, and again $w_C(D') = w_C(D)$.
	
		As for R4, clearly $w(D'_c) = w(D_c)$ if $c$~does not contain the edge~$e$ crossing over/under the relevant vertex~$v$, so assume that $c \ni e$. Next, note that $c$~contains either none, or exactly two among the half-edges incident to~$v$. In the first case, there is again clearly no effect on~$w(D_c)$. If $c$~contains the two half-edges at~$v$ that gain or lose a crossing, then on~$D_c$ this is indistinguishable from an R2 move. Otherwise, $D_c$~merely has one crossing displaced. Either way, $w(D'_c) = w(D_c)$ and we conclude $w_C(D') = w_C(D)$.
\end{proof}

\begin{cor}[Realizing tuples]
	Let $D$~be a diagram of a spatial graph~$\Gamma$ of type~$G$, let $C$~be a set of cycles in~$G$, and let $w\in w_C(D) +  C^*(\ZZ^{\E(G)})$. Then $\Gamma$~has a diagram~$D'$ with $w_C(D') = w$.
\end{cor}
\begin{proof}
	If $w - w_C(D) = C^*((n_e)_{e\in \E(G)})$, we modify~$D$ by performing $|n_e|$~moves of type R1 at each edge~$e$, adding crossings of the same sign as~$n_e$. Lemma~\ref{lem:R14}~(1) guarantees the resulting diagram~$D'$ has~$w$ as its writhe tuple.
\end{proof}

In describing the effect of the move~R5, we will use the same notation~$C^*$ for the extended map
\[C^* \colon \left(\tfrac 12 \ZZ\right)^{\E(G)} \to \left(\tfrac 12 \ZZ\right)^C,\]
where $\frac 12 \ZZ := \{\ldots , -1, -\frac 12 ,0, \frac12, 1, \ldots\}$ is the free $\ZZ$-module of half-integers.

\begin{lem}[Move R5]\label{lem:R5}
	Let $D,D'$ be diagrams of $3$-regular spatial graphs of type~$G$, and let~$C$ be a set of cycles of~$G$. Suppose $D'$~differs from~$D$ by an $\mathrm{R5}$ move at the vertex~$v$. Let $e_1, e_2 \in \E(G)$ be the edges at~$v$ that gain or lose a crossing, and $e_3$~the remaining edge.
	
	Then 
	\[w(D') = w(D) + C^*\left( \tfrac \varepsilon 2 (e_1 + e_2 - e_3)\right),\]
	where $\varepsilon = 1$ if the move adds a crossing, and $\varepsilon = -1$ if it removes a crossing.
\end{lem}
 
 Note that $e_1, e_2, e_3$ are not necessarily distinct, as $G$~might have a loop at~$v$.
 
\begin{proof}
Each~$c\in C$ contains either none, or exactly two of the half-edges at~$v$, and
\[w(D'_c) - w(D_c) = \begin{cases}
	\varepsilon & \text{if $e_1, e_2 \in \E(c)$,}\\
	0 &  \text{otherwise.}
\end{cases} \tag{$\star$}\]

For the element
$u := \tfrac \varepsilon 2( e_1 +  e_2 -  e_3)$,
we have
\[ C^*(u) = \frac \varepsilon 2 \biggl( \sum_{c \ni e_1} c+ \sum_{c\ni e_2} c- \sum_{c \ni e_3}c\biggr).\]
In this sum, the cycles~$c$ containing both $e_1$~and~$e_2$ (hence not~$e_3$) have coefficient $\frac \varepsilon 2 (1+1)= \varepsilon$, those containing $e_1$~and~$e_3$ have coefficient $\frac \varepsilon 2 (1-1) = 0$, and similarly for those containing $e_2$~and~$e_3$. Obviously if $c$~contains none among $e_1, e_2, e_3$, its coefficient is also~$0$. In all cases, this is precisely the value of~$(\star)$, and so
\[w_C(D') - w_C(D) = C^*(u).\qedhere\]
\end{proof}

\begin{proof}[Proof of Theorem~\ref{thm:main}]
	By Theorem~\ref{thm:reidemeister}, $D_2$~differs from~$D_1$ by a sequence of Reidemeister moves. 
	Thus, Lemmas \ref{lem:R14}~and~\ref{lem:R5} combined show that $w_C(D_2) - w_C(D_1)$ is a sum of elements of the form $\varepsilon C^*(e) \in  C^*(\ZZ^{\E(G)})$, and of the form $C^*(\frac \varepsilon 2(e_1 + e_2 - e_3))$ with $e_1, e_2, e_3$ the edges at a vertex. 
	The condition on induced cyclic orders guarantees that, moreover, an even number of R5 moves is applied at each vertex, so the summands of the second type can be grouped into pairs associated to the  same vertex.
	
	Now, for every two elements
	\[u = \tfrac \varepsilon 2(e_1 + e_2 - e_3), \qquad u' = \tfrac {\varepsilon'} 2(e'_1 + e'_2 - e'_3)\]
	with $\varepsilon, \varepsilon'\in \{\pm1\}$ and $\{e_1, e_2, e_3\} = \{e'_1, e'_2, e'_3\}$, the sum $u+u'$ has coefficients in~$\{-1, 0, 1\}$, so $u + u' \in \ZZ^{\E(G)}$. Hence, the joint contribution of each pair has the form
	\[C^*(u + u') \in   C^*(\ZZ^{\E(G)}). \qedhere\]
\end{proof}

\begin{ex}\label{ex:mainexample}
	The diagrams for the spatial graphs~$\Gamma_1, \Gamma_2$ from Example~\ref{ex:writhetuple} have matching cyclic orders of the half-edges at every vertex. Their writhe tuples associated to the set~$C$ of $6$-cycles differ by
	\[w_C(D_2) - w_C(D_1) = (0,0,-2, 2, 0, 2).\]
	One now checks that the map~$C^*$, whose matrix is shown in Example~\ref{ex:cyclemap}, does not have this tuple in its image. A quick way to see this is by noting that all columns of that matrix have entries adding up to~$4$ (in other words, each edge of~$\K$ is in exactly~$4$ length-$6$ cycles). Yet the entry sum of~$w_C(D_2) - w_C(D_1)$ is not a multiple of~$4$. Thus, by Theorem~\ref{thm:main}, we conclude $\Gamma_1 \not \cong\Gamma_2$.
	
	We point out that the same argument could have been carried out using as~$C$ the set of $4$-cycles instead.
\end{ex}

Example~\ref{ex:mainexample} can be vastly generalized:

\begin{ex}\label{ex:moebius}
	Given $n \in \NN_{\ge1}$, the \textbf{Möbius ladder} with $n$~rungs~$\M_n$  is the $3$-regular graph with vertex set $\ZZ / {2n}$, and
	\begin{itemize}
		\item for each $k\in \ZZ/2n$, an edge~$e_k$ connecting~$k$ to~$k+1$,
		\item for each $l\in\ZZ / n$, an edge $r_l$ connecting the two vertices that have~$l$ as their mod-$n$ residue
	\end{itemize}
	(with some choice of edge orientations, which will be irrelevant). See Figure~\ref{fig:moebius}.
	
	\begin{figure}[h]
		\centering
		\def \svgwidth{\linewidth}
		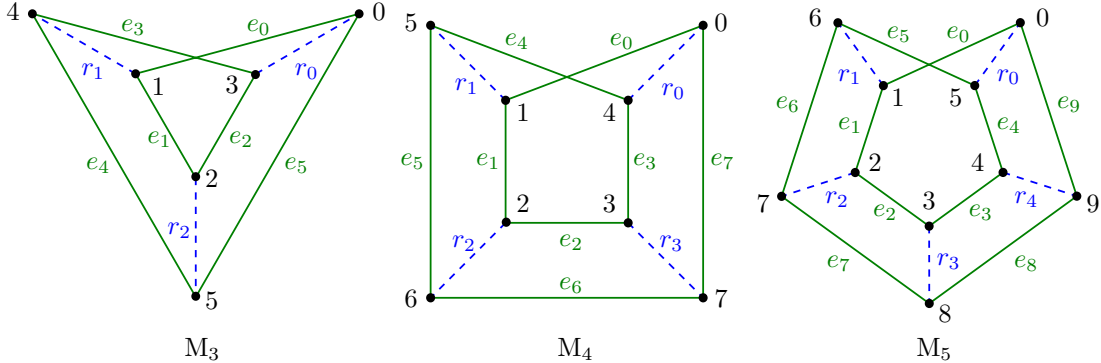
		\caption{The Möbius ladders $\M_3, \M_4$ and  $\M_5$.}
		\label{fig:moebius}
	\end{figure}
	
	For each $n\ge 1$ and odd integer~$m$, we consider the spatial graph~$\Gamma_n^m$ of type~$\M_n$ with diagram~$D_n^m$ as in Figure~\ref{fig:Gamma43}: there are exactly $|m|$~crossings, all between  $e_0$~and~$e_n$, and their sign matches the sign of~$m$ (if $e_0, e_n$ are oriented, respectively, from $0$~to~$1$ and from $n$~to~$n+1$).

	\begin{figure}[h]
		\centering
		\def \svgwidth{0.3\linewidth}
		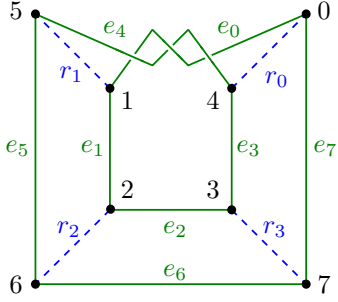
		\caption{The diagram $D_4^3$ of~$\Gamma_4^3$.}
		\label{fig:Gamma43}
	\end{figure}
	
	We will use Theorem~\ref{thm:main} to prove that if $n\ge 3$, then $\Gamma_n^m \not \cong \Gamma_n^{m'}$ whenever $m\neq m'$. (For $n=1$ or $n=2$ we recover, respectively, the theta-graph and the complete graph on 4~vertices~$\mathrm K_4$. In either case, one easily checks that for every odd $m, m'$, we have $\Gamma_n^m \cong \Gamma_n^{m'}$.) 
	
	Let us then fix $n\ge 3$ once and for all. We consider:
	\begin{itemize}
		\item the set~$C_4$ of $4$-cycles obtained by choosing $l\in \ZZ / n$, and taking the edges~$r_{l}, r_{l+1}$ alongside the two~$e_k$ with $k\equiv l \mod{n}$,
		\item the set~$C_{+}$ of $(n+1)$-cycles obtained by choosing $k\in \ZZ / 2n$ and taking the edge~$r_k$, along with $e_{k}, e_{k+1}, \ldots, e_{k+n-1}$,
		\item the $2n$-cycle~$\gamma$ comprised of the edges~$e_k$ over all $k\in \ZZ / 2n$.
	\end{itemize}
	
	We now put $C:= C_4 \cup C_+\cup \{\gamma\}$.
	For each odd~$m$, there are precisely two non-zero entries in~$w_C(D_n^m)$: the $4$-cycle containing $e_0$~and~$e_n$ has coordinate~$-m$, and $\gamma$~has coordinate~$m$. We illustrate the  system of equations relevant for comparing $\Gamma_n^m, \Gamma_n^{m'}$ in the $n=4$ case:
	
	\[\begin{array}{c|cccccccccccc|c}
		&e_0&e_1&e_2&e_3&e_4&e_5&e_6&e_7&r_0&r_1&r_2&r_3& w_C(D_n^{m'})-w_C(D_n^m)\\ \hline
		\multirow{4}*{$C_4$}&1&&&&1&&&&1&1&&& -(m'-m)\\
		&&1&&&&1&&&&1&1&&0\\
		&&&1&&&&1&&&&1&1&0\\
		&&&&1&&&&1&1&&&1&0\\ \hline
		\multirow{8}*{$C_{+}$}&1&1&1&1&&&&&1&&&&0\\
		&&1&1&1&1&&&&&1&&&0\\
		&&&1&1&1&1&&&&&1&&0\\
		&&&&1&1&1&1&&&&&1&0\\
		&&&&&1&1&1&1&1&&&&0\\
		&1&&&&&1&1&1&&1&&&0\\
		&1&1&&&&&1&1&&&1&&0\\
		&1&1&1&&&&&1&&&&1&0\\ \hline
		\gamma &1&1&1&1&1&1&1&1&&&&&m'-m
	\end{array}\]
	
	We will show that this system has no solution if $m\neq m'$, and therefore, by Theorem~\ref{thm:main}, we conclude $\Gamma_n^m \not \cong \Gamma_n^{m'}$. For each $l \in \ZZ/n$, the sum of the two ``$C_+$'' equations involving~$r_l$, minus the ``$\gamma$'' equation, shows that the $r_l$-coordinate of any solution is $-\frac{m'-m}2$. Using this information, the sum of all ``$C_4$'' equations (of which there are~$n$) tells us that the $e_k$-coordinates add up to~$(n-1)(m'-m)$.
	From the ``$\gamma$'' equation, we then deduce $(n-2)(m'-m) =0$, which, as $n\neq 2$, implies $m'=m$.
\end{ex}

\section{Additional remarks}\label{sec:extras}

\subsubsection*{3-regularity}

The $3$-regularity assumption in Theorem~\ref{thm:main} is essential, as the next example shows.

\begin{ex}
	The diagrams $D_1, D_2$ in Figure~\ref{fig:4edges} represent isotopic spatial graphs and induce the same cyclic orders on half-edges at each vertex. They are however not $3$-regular. 

	\begin{figure}[h]
		\centering
		\def \svgwidth{0.6\linewidth}
		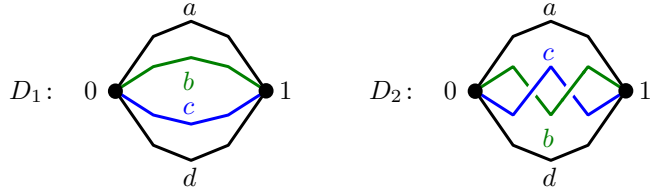

		\caption{Diagrams for two isotopic spatial graphs.
		} 
		\label{fig:4edges}
	\end{figure}
		The reader can easily check that for $C$~the set of all $2$-cycles (of which there are~$6$), we have $w_C(D_2) - w_C(D_1) \not \in C^*(\ZZ^{\E(G)})$
\end{ex}

\subsubsection*{Choosing $C$}

In applying Theorem~\ref{thm:main} to distinguish spatial graphs of type~$G$ given as diagrams $D_1, D_2$, one must settle on a set~$C$ of cycles. Note that the map~$C^*$ satisfies the following naturality property: given a subset $C'\subseteq C$, the corresponding map $(C')^*$ is the composition
\[\ZZ^{\E(G)} \xrightarrow{C^*} \ZZ^C \onto \ZZ^{C'},\] 
where the map to the right is the canonical projection. Each writhe tuple~$w_{C'}(D_i)$ is also the projection of~$w_C(D_i)$.
Thus, if $w_{C'}(D_2) - w_{C'}(D_1) \not \in (C')^*(\ZZ^{\E(G)})$, then also $w_{C}(D_2) - w_{C}(D_1) \not \in C^*(\ZZ^{\E(G)})$. In other words, the method is, unsurprisingly, more powerful if the larger set~$C$ is used.

\subsubsection*{Invisible crossing changes}

A \textbf{crossing change} on a spatial graph diagram is the operation of flipping the data at a crossing of which strand crosses over which.
The following proposition indicates a situation where Theorem~\ref{thm:main} has no chance of obstructing the existence of an isotopy between spatial graphs.

\begin{prop}[Invisible crossing changes]\label{prop:crossingchanges}
	Let $D, D'$ be diagrams of $3$-regular spatial graphs of type~$G$, and let $C$~be a set of cycles of~$G$. If $D$ and~$D'$ differ only by a crossing change between strands in the same edge, or in adjacent edges, then \[w_C(D') - w_C(D) \in C^*(\ZZ^{\E(G)}).\]
\end{prop}
\begin{proof}
	Suppose first that both strands at the relevant crossing~$p$ of~$D$ are in the same edge~$e$. The sign of~$p$ is the same in the diagram~$D_c$ of every $c\in C$ containing~$e$, and the crossing change alters that sign in every~$c$, adding the same value $2\varepsilon = \pm 2$ to~$w(D_c)$ for every~$c$ containing~$e$. All other cycles are  unaffected, so $w_C(D') - w_C(D) = C^*(2\varepsilon e)$. (In other words, from the perspective of the writhe tuple, this crossing change is indistinguishable from performing two R1 moves at~$e$.)
	
	If the crossing change happens between two different edges~$e_1, e_2$ meeting at~$v$, then the cycles~$c\in C$ for which $w(D_c)$ gets affected are precisely the ones containing both $e_1$~and~$e_2$. Since $e_1, e_2$ are adjacent, the sign of~$p$ in~$D_c$ is the same for every $c\in C$, and the crossing change has the effect of adding the same value $2\varepsilon = \pm2$ to $w(D_c)$ for $c$~containing $e_1$~and~$e_2$. As far as the writhe tuple is concerned, this is the same as applying two R5 moves at~$v$, or its mirrored version, between $e_1$~and~$e_2$. Hence, if $D^\circ$~is the diagram obtained from~$D$ by performing these two moves (instead of the crossing change), we have $w_C(D^\circ) = w_C(D')$. Since~$D^\circ$~and~$D$ represent isotopic spatial graphs, and they have matching cyclic orders, we conclude by Theorem~\ref{thm:main} that
	\[w_C(D') - w_C(D) = w_C(D^\circ) - w_C(D) \in C^*(\ZZ^{\E(G)}).\qedhere\]
\end{proof}

\section{Connection to the Wu invariant}\label{sec:wu}

In 1995, Taniyama gave a classification of spatial graphs without degree restrictions ``up to homology'' (a coarser equivalence than isotopy) \cite{Tan95}, by specializing an invariant due to Wu for embedded CW-complexes in Euclidean spaces \cite{Wu59}. For another exposition, see for example the paper of Flapan-Fletcher-Nikkuni \cite[Section~2]{FFN14}. In this section, we briefly present Taniyama's very explicit description of the Wu invariant for spatial graphs, and connect it to Theorem~\ref{thm:main}.

Given a graph~$G$, consider the free $\ZZ$-module~$Z(G)$ spanned by the unordered pairs $E^{e,f} = E^{f,e} := \{e,f\}$ of edges that do not share a vertex.
Then define, for each vertex~$v$ and edge $e$ not incident to~$v$, the following element of~$Z(G)$:
\[V^{v,e} := \sum_{\substack{\text{$f$~incident to~$v$}\\ \text{$f$ disjoint from $e$}}} \delta^{v,f} E^{e,f}, \quad \text {where~} \delta^{v,f} := \begin{cases}
	1& \text{if $s_G(f) = v$,}\\
	-1& \text{if $t_G(f) = v$.}
\end{cases}\]
Note that in this sum, a loop~$f$ at~$v$ is counted once with each sign, so the coefficient of~$E^{e,f}$ vanishes.
We let $B(G) \subseteq Z(G)$ be the submodule generated by all~$V^{v,e}$, and define the \textbf{linking module} of~$G$ to be
\[L(G) := Z(G) / B(G).\]

The Wu invariant of a spatial graph~$\Gamma$ is then an element of~$L(G)$, which can be read off of a diagram~$D$ as follows:  for each pair of edges $e,f$ of~$G$, let $\ell_D(e,f) = \ell_D(f,e)$ be the signed number of crossings between $e,f$ (read off using the orientations of~$e,f$). The \textbf{Wu invariant} of~$\Gamma$ is
\[\L(\Gamma) = \biggl[ \sum_{\text{$e, f$~disjoint}}  \ell_D(e,f) E^{e,f} \biggr] \in L(G).\]
Tanayama shows that this equality (which we will take as a definition) recovers Wu's invariant, and is therefore independent of the diagram~$D$ \cite[Proposition~2.1]{Tan95}.
Note that crossings in~$D$ between edges that share a vertex do not contribute to the sum, which should remind us of Proposition~\ref{prop:crossingchanges}!

For our purposes, it will be convenient to express the linking module as a modified quotient
\[L(G) = \tilde Z(G) / \tilde B(G),\]
where $\tilde Z(G)$~differs from $Z(G)$ in that it is has as a $\ZZ$-basis variables $E^{e,f}$ for \emph{all} unordered pairs of edges~$e,f$ (including when $e,f$ share a vertex, or even coincide). Accordingly, we also define~$\tilde B(G)\subseteq \tilde Z(G)$ as the submodule generated by the
\[\tilde V^{v,e} := \sum_{\text{$f$~incident to~$v$}} \delta^{v,f} E^{e,f},\]
taken over all vertices~$v$ and edges~$e$ not incident to~$v$, along with all the variables~$E^{e,f}$ whenever $e$~and~$f$ share a vertex. This results in a quotient that is canonically isomorphic to~$L(G)$ as defined above, and the Wu invariant of a spatial graph~$\Gamma$ with diagram~$D$ can be expressed as a sum over \emph{all} edges:
\[\mathcal L(\Gamma) = \bigg[ \sum_{e,f\in \E(G)} l_D(e,f)E^{e,f} \bigg].\]

From now on, assume $G$~is $3$-regular. Before presenting the main result of this section, we need to introduce the following notation:
\begin{itemize}
	\item For each $v\in \V(G)$, with incident edges $e,f,g$ (loops counted twice), define
	\[u_v := \tfrac 12 (e + f + g) \in (\tfrac 12 \ZZ)^{\E(G)}.\]
	\item Denote by~$\C_G \subseteq  (\tfrac 12 \ZZ)^{\E(G)}$ the submodule spanned by $\E(G) \cup \{u_v \mid v\in \V(G)\}$.
	Recall the observation at the heart of Lemma~\ref{lem:R5}, that when $C^*$~is regarded as the extended map $(\tfrac 12\ZZ)^{\E(G)} \to (\tfrac 12 \ZZ)^{C}$, we have
	\[C^*(u_v - g) = \sum_{c \ni e,f} c \in \ZZ^C.\tag{$\dagger$}\]
	In particular, $C^*(\C_G) \subseteq \ZZ^C$.
	\item Given edges~$e,f$ in a cycle~$c$ of~$G$, put
	\[\varepsilon^{e,f}_c := \begin{cases}
		1&\text{if $e,f$~have matching orientation along~$c$,}\\
		-1& \text{otherwise.}
	\end{cases}\]
\end{itemize}

With this notation in place, we can now clarify how the invariance of Writhe tuple cosets given by Theorem~\ref{thm:main} can be seen as a consequence of the Wu invariant.

\begin{prop}[Writhe tuple from the Wu invariant]\label{prop:wu}
	If $G$~is a $3$-regular graph, then the $\ZZ$-module homomorphism
	\begin{align*}
		\varphi\colon \tilde Z(G)& \to \ZZ^C\\
		E^{e,f} &\mapsto \sum_{c \ni e,f} \varepsilon^{e,f}_c c
	\end{align*}
	descends to a map $\phi\colon L(G) \to \ZZ^C / C^*(\C_G)$.
	
	Moreover, if $D$~is a diagram of a spatial graph~$\Gamma$ of type~$G$, then $\phi(\L(\Gamma)) = [w_C(D)]$.
\end{prop}

\begin{proof}	
	For the first part, we need to check that $\varphi(\tilde V^{v,e}) \in C^*(\C_G)$ for every vertex~$v$ and disjoint edge~$e$, and that $\varphi(E^{e,f}) \in C^*(\mathcal C_G)$ whenever $e,f$~share a vertex.
	
	Let us first deal with the latter. For $e=f$, we see
	\[\varphi(E^{e,e}) = \sum_{c\ni e} c  = C^*(e),\]
	and if $e$ and~$f$ are distinct edges sharing a vertex~$v$, then $\varepsilon^{e,f}_c = -\delta_1\delta_2$ (independently of~$c$), so
	\[\varphi(E^{e,f})= -\delta_1\delta_2 \sum_{c\ni e,f} c \overset{(\dagger)}{=}  -\delta_1\delta_2 C^*(u_v-g),\]
	where $g$~is the third edge at~$v$.
	
	We next compute~$\varphi(\tilde V^{v,e})$ for every vertex~$v$ and disjoint edge~$e$. Denoting by~$e_1, e_2, e_3$ the edges at~$v$ appearing in the definition of $\tilde V^{v,e}$, we abbreviate $\delta^{v, e_i} =: \delta_i$. The key is the following observation:
	for every cycle~$c$ containing $e, e_1$~and~$e_2$, we have
	\[\delta_1 \varepsilon_c^{e_1, e} = - \delta_2 \varepsilon_c^{e_2, e}\tag{$\star$}\]
	(and similarly for the other pairs of indices). To see this, one should first note that, since $e_1$~and~$e_2$ are in a cycle with~$e$, they cannot be loops, so they do not appear twice in the sum defining $\tilde V^{v,e}$, and hence $e_1 \neq e_2$. One can then easily see that
	\[\varepsilon_c^{e_1, e} \varepsilon_c^{e_2,e} = \varepsilon_c^{e_1,e_2} = -\delta_1 \delta_2,\]
	from which ($\star$) follows.
	
	We now see
		\[ \varphi(\tilde V^{v,e}) = \delta_1 \varphi(E^{e_1, e}) + \delta_2 \varphi(E^{e_2, e}) + \delta_3 \varphi(E^{e_3, e}).\]
	Expanding the summand coming from~$e_1$, we obtain
	\begin{align*}
		\delta_1 \varphi(E^{e_1, e}) &= \delta_1 \sum_{c \ni e_1, e} \varepsilon_c^{e_1,e}c\\
		& = \delta_1 \sum_{c\ni e_1, e_2, e} \varepsilon_c^{e_1,e}c \, + \, \delta_1 \sum_{c\ni e_1, e_3, e} \varepsilon_c^{e_1,e}c.
	\end{align*}
	Using ($\star$), this can be rewritten as
	\[- \delta_2 \sum_{c\ni e_1, e_2, e} \varepsilon_c^{e_2,e}c \, - \, \delta_3 \sum_{c\ni e_1, e_3, e} \varepsilon_c^{e_3,e}c.\]
	The first summand thus cancels one of the summands coming from~$e_2$, and the second cancels one summand coming from $e_3$. By the same reasoning, the remaining summands coming from~$e_2, e_3$ cancel each other, and we conclude $\varphi(\tilde V^{v,e}) = 0$.
	
	For the second part, we see that $\phi(\L(\Gamma))$ is represented by

	\[\sum_{e, f\in \E(G)}  \ell_D(e,f) \varphi (E^{e,f}) = \sum_{e,f \in \E(G) }  \ell_D(e,f)  \sum_{c\ni e,f}\varepsilon_c^{e,f} c,\]
	where it is clear that for each $c\in C$, the $c$-coordinate is
	\[\sum_{c\ni e,f} \varepsilon_c^{e,f} \ell_D(e,f) = w(D_c).\qedhere\]
\end{proof}

We finish by drawing the reader's attention to the fact that the residue class of~$w_C(D)$ in Proposition~\ref{prop:wu} is with respect to~$C^*(\C_G)$, which is in general larger than the submodule~$C^*(\ZZ^{\E(G)})$ featured in our main result, Theorem~\ref{thm:main}. This might suggest that Theorem~\ref{thm:main} refines the criterion given by the Wu invariant. However, this is merely an artifact of the linking module being defined without a fixed choice of cyclic orders of half-edges at vertices, whereas our main result requires one such choice. We leave it as an exercise for the interested reader to explain how one can tweak the definition of~$L(G)$ and~$\L(\Gamma)$ to a modified linking module~$L'(G)$ and an element $\L'(\Gamma)\in L'(G)$ that is only well defined up to moves R1--R4 and an \emph{even} number of R5 moves at each vertex. One can then define a map~$\phi'\colon L'(G) \to \ZZ^C / C^*(\ZZ^{\E(G)})$ analogous to the above~$\phi$.

\printbibliography

@Article{FFN14,
 Author = {Flapan, Erica and Fletcher, Will and Nikkuni, Ryo},
 Title = {Reduced {Wu} and generalized {Simon} invariants for spatial graphs},
 FJournal = {Mathematical Proceedings of the Cambridge Philosophical Society},
 Journal = {Math. Proc. Camb. Philos. Soc.},
 ISSN = {0305-0041},
 Volume = {156},
 Number = {3},
 Pages = {521--544},
 Year = {2014},
 %Language = {English},
 DOI = {10.1017/S0305004114000073},
 Keywords = {05C10,57M15},
 %zbMATH = {6293651},
 %Zbl = {1287.05027}
}

@article{FMQS24,
    title={Canonical decompositions and algorithmic recognition of spatial graphs},
    DOI={10.1017/S0013091524000087},
    fjournal={Proceedings of the Edinburgh Mathematical Society},
    journal={Proc. Edinburgh Math. Soc.},
    author={Friedl, Stefan and Munser, Lars and Quintanilha, José Pedro and Santos Rego, Yuri},
    year={2024},
    pages={1--43}
}

@Article{Kau89,
 Author = {Kauffman, Louis H.},
 Title = {Invariants of graphs in three-space},
 FJournal = {Transactions of the American Mathematical Society},
 Journal = {Trans. Am. Math. Soc.},
 ISSN = {0002-9947},
 Volume = {311},
 Number = {2},
 Pages = {697--710},
 Year = {1989},
 %Language = {English},
 DOI = {10.2307/2001147},
 Keywords = {57M25,57Q35},
 zbMATH = {4100185},
 Zbl = {0672.57008}
}

@book{RS72,
    AUTHOR = {Rourke, Colin P. and Sanderson, Brian J.},
     TITLE = {Introduction to piecewise-linear topology},
      NOTE = {Ergebnisse der Mathematik und ihrer Grenzgebiete, Band 69},
 PUBLISHER = {Springer-Verlag, New York-Heidelberg},
      YEAR = {1972},
	DOI = {10.1007/978-3-642-81735-9},
}

@Article{Tan95,
 Author = {Taniyama, Kouki},
 Title = {Homology classification of spatial embeddings of a graph},
 FJournal = {Topology and its Applications},
 Journal = {Topology Appl.},
 ISSN = {0166-8641},
 Volume = {65},
 Number = {3},
 Pages = {205--228},
 Year = {1995},
 %Language = {English},
 DOI = {10.1016/0166-8641(95)00008-5},
 Keywords = {57M25,05C10},
 %zbMATH = {823461},
 %Zbl = {0843.57012}
}

@Article{Wu59,
 Author = {Wu, Wen-ts{\"u}n},
 Title = {On the isotopy of a finite complex in a {Euclidean} space. {I}, {II}},
 FJournal = {Science Record. New Series},
 Journal = {Sci. Record (N.S.)},
 ISSN = {0452-2176},
 Volume = {3},
 Pages = {342--347, 348--351},
 Year = {1959},
 %Language = {English},
 Keywords = {57-XX},
 %zbMATH = {3195031},
 %Zbl = {0119.38701}
}

\textsc{Stefan Friedl}, Universität Regensburg, Fakultät für Mathematik, 93053 Regensburg, Germany.
\textit{E-mail: }
\texttt{\href{mailto:stefan.friedl@mathematik.uni-regensburg.de}{stefan.friedl@mathematik.uni-regensburg.de}}\medskip

\textsc{Tejas Kalelkar}, IISER Pune, Dr. Homi Bhabha Road,
Pune 411008, India. \textit{E-mail: }
\texttt{\href{mailto:tejas@iiserpune.ac.in}{tejas@iiserpune.ac.in}}\medskip

\textsc{José Pedro Quintanilha}, Institut für Mathematik IMa, Im Neuenheimer Feld 205, 69120 Heidelberg, Germany. \textit{E-mail: }
\texttt{\href{mailto:jquintanilha@mathi.uni-heidelbeg.de}{jquintanilha@mathi.uni-heidelbeg.de}}

\end{document}